\documentclass[12pt]{article}

\input{settings} 

\title{\sc Regular decompositions of finite root systems and simple Lie algebras}
\author{\sc Stepan Maximov}
\date{\sc \today}

\begin{document}
\maketitle

\begin{abstract}
    Let $\mathfrak{g}$ be a finite-dimensional simple Lie algebra over an algebraically closed field of characteristic 0.
    In this paper we classify all regular decompositions of $\fg$ and its irreducible root system $\Delta$.
    
    A regular decomposition is a decomposition $\mathfrak{g} = \mathfrak{g}_1 \oplus \dots \oplus \mathfrak{g}_m$, where each $\fg_i$ and $\fg_i \oplus \fg_j$ are regular subalgebras. Such a decomposition induces a partition of the corresponding root system, i.e.\ $\Delta = \Delta_1 \sqcup \dots \sqcup \Delta_m$, such that all $\Delta_i$ and $\Delta_i \sqcup \Delta_j$ are closed. 
    
    Partitions of $\Delta$ with $m=2$ were known before. In this paper we prove that the case $m \ge 3$ is possible only for systems of type $A_n$ and describe all such partitions in terms of $m$-partitions of $(n+1)$.
    These results are then extended to a classification of regular decompositions of $\fg$.
\end{abstract}

\section{Introduction}
Let $\fg$ be a finite-dimensional Lie algebra with a Lie group $G$. 
Decompositions 
\begin{equation}%
\label{eq:splitting_subalgebras}
   \fg = \fg_1 \oplus \fg_2 
\end{equation}
of $\fg$ into a direct sum of two subalgebras give rise to many non-linear dynamical systems, that are Liouville integrable and can be reduced to a factorization problem inside $G$ \cite{adler_moerbeke_vanhaecke,semenov2008integrable,rudolph2012differential}.
When $\fg$ is additionally equipped with a non-degenerate symmetric ad-invariant bilinear form, such that both $\fg_1$ and $\fg_2$ become Lagrangian subalgebras, we get a Manin triple $(\fg, \fg_1, \fg_2)$, that defines a Lie bialgebra structure on $\fg_1$ (and $\fg_2$). 
Lie bialgebras are infinitesimal objects related to Poisson-Lie groups, as Lie algebras related to Lie groups \cite{kosmann-schwarzbach,laurent2012poisson}. 
These objects were introduced in \cite{drinfeld_hamiltonian_structures} and they play an import role in quantum group theory \cite{drinfeld_quantum_groups,etingof_schiffmann}.

Decompositions \cref{eq:splitting_subalgebras} can be naturally generalised to subalgebra decompositions
\begin{equation}%
\label{eq:m_decomposition}
    \fg = \fg_1 \oplus \fg_2 \oplus \dots \oplus \fg_m,
\end{equation}
such that $\fg_i \oplus \fg_j$ is again a subalgebra of $\fg$.
It was shown in \cite{golubchik_sokolov_yang_baxter} that any such decomposition gives rise to another subalgebra decomposition
\begin{equation}%
\label{eq:Laurent_decomposition}
    \fg(\!(x)\!) = \fg[\![x]\!] \oplus W,
\end{equation}
where $\fg(\!(x)\!)$ and $\fg[\![x]\!]$ are the Lie algebras of Laurent and Taylor series respectively.
Algebra decompositions of the form  \cref{eq:Laurent_decomposition} correspond to so-called classical generalized formal $r$-matrices \cite{abedin_maximov_stolin_quasibialgebras, abedin_maximov_stolin_zelmanov}.
These generalized $r$-matrices can again be used to construct integrable systems, for example generalized Gaudin models \cite{skrypnyk_spin_chains}.

Classification of decompositions \cref{eq:splitting_subalgebras} even for simple complex Lie algebras is a representation wild problem.
For this reason, to obtain some sensible classification, we introduce more constraints on decompositions.
More precisely, we look at a finite-dimensional simple Lie algebra $\fg$ over an algebraically closed field $F$, with a fixed Cartan subalgebra $\fh$ and the corresponding root space decomposition
\begin{equation}%
\label{eq:root_spaces}
    \fg = \fh \bigoplus_{\alpha \in \Delta} \fg_\alpha.
\end{equation}
We consider regular decompositions of $\fg$, i.e.\ decompositions that respect the root structure \cref{eq:root_spaces}.

An \emph{$m$-regular decomposition of $\fg$} is a decomposition
\begin{equation}%
\label{eq:regular_decomp}
    \fg = \bigoplus_{i = 1}^m \fg_i, \ m \ge 2
\end{equation}
satisfying the following restrictions:
\begin{enumerate}
    \item all $\fg_i$ as well as $\fg_i \oplus \fg_j$ are Lie subalgebras of $\fg$;
    \item each $\fg_i$ has the form $ \fs_i \bigoplus_{\alpha \in \Delta_i} \fg_\alpha$ for some subspace $\fs_i \subseteq \fh$ and some subset $\Delta_i \subseteq \Delta$.
\end{enumerate}
Subalgebras of $\fg$ of the form described in 2.\ are called regular, motivating the name.

It is clear from the definition above, that by forgetting the Cartan part of $\fg$, we obtain a partition of the corresponding irreducible root system
$$
  \Delta = \bigsqcup_{i=1}^m \Delta_i,
$$
such that all $\Delta_i$ and $\Delta_i \sqcup \Delta_j$ are closed under root addition. 
By analogy with \cref{eq:regular_decomp}
we call such a partition \emph{$m$-regular}.


Regular partitions $\Delta = \Delta_1 \sqcup \Delta_2$ are precisely the partitions $\Delta = \Delta_1 \sqcup (\Delta \setminus \Delta_1)$, such that both $\Delta_1$ and its complement $\Delta_1^c \coloneqq (\Delta \setminus \Delta_1)$ are closed.
Subsets $\Delta_1 \subseteq \Delta$ with this property are called \emph{invertible}. 
Invertible subsets of finite irreducible root systems were classified in \cite{dokovic_check_hee} up to the action of the corresponding Weyl group $W(\Delta)$. 

We prove in \cref{prop:2_partititons} that any $2$-regular partition of $\Delta$ can be extended (non-uniquely) to a $2$-regular decomposition of the corresponding simple Lie algebra $\fg(\Delta)$. 
Non-uniqueness comes essentially from the non-uniqueness of vector space factorization $\fh = \fh_1 \oplus \fh_2$.
This gives a description of all $2$-regular decompositions of $\fg$ up to a choice of such a decomposition of $\fh$.

Observe that the results of \cite{dokovic_check_hee} cannot be applied inductively to $\Delta = \Delta_1 \sqcup \Delta_1^c$, because neither of the two sets is symmetric and hence neither of them is a root system.
One of the main results of this paper is the following description of all $(m \ge 3)$-regular partitions.

\begin{maintheorem}%
\label{mainthm:root_partitions}
    Let $\Delta = \sqcup_1^m \Delta_i$ be a regular partition of a finite irreducible root system into $m \ge 3$ parts. 
    Then 
    \begin{enumerate}
        \item $\Delta$ is necessarily of type $A_n$, $n \ge 2$;
        \item Up to swapping positive and negative roots, re-numbering elements $\Delta_i$ of the partition and action of $W(A_n)$, there is a unique maximal $(n+1)$-partition of $A_n$ with
        $$
    \Delta_i = \{-\beta_i + \beta_j \mid 0 \le i \neq j \le n \}, \ \ 0 \le i \le n,
    $$
    where $\beta_0 = 0$, $\beta_i = \alpha_1 + \dots + \alpha_i$, $1 \le i \le n$ and $\pi = \{ \alpha_1 , \dots, \alpha_n\}$ are simple roots of $\Delta$;
    \item Any other $(m < n+1)$-regular partition is obtained from the maximal one above by combining several subsets $\Delta_i$ together;
    \item Up to equivalences mentioned above, all $m$-regular partitions are described by $m$-partitions $\lambda = (\lambda_1, \dots, \lambda_m)$ of $n+1$. 
    \end{enumerate}
\end{maintheorem}

\begin{remark}
  Non-existence of $(m\ge 3)$-regular partitions of root systems of other types is proven case-by-case.  
  In all cases, the obstruction arises because of coefficients $k_i > 1$ in the decomposition of the maximal root into a sum of simple roots.
  Unfortunately, we were not able to find a unified approach.
\end{remark}

The statement of \cref{mainthm:root_partitions} implies immediately, that $m$-regular decompositions of a simple Lie algebra $\fg$ is possible only in the case $\fg \cong \mathfrak{sl}(n+1, F)$, $n \ge 2$.
We prove in \cref{prop:m_partititons} that any $m$-regular partition of $\Delta$ can be extended to an $m$-regular decomposition of $\fg$.

In an $m$-regular decomposition of $\fg$ the Cartan part $\fh$ is distributed among $\fg_i$'s in a very restrictive way. The larger $m$ we take, the more restrictions we get.
We say that a regular decomposition of $\fg$ is of \emph{type $(m,k)$} if it is an $m$-regular decomposition with exactly $k$ summands having non-zero Cartan parts $s_i$.

In the following theorem, we view $\mathfrak{sl}(n+1, F)$ as the algebra of traceless $(n+1)\times (n+1)$-matrices over $F$ and we write $E_{i,j}$ for the $(0,1)$-matrix with $1$ in position $(i,j)$, and put $H_i \coloneqq E_{1,1} - E_{i+1, i+1}$ for $0 \le i \le n$.

\begin{maintheorem}%
\label{mainthm:classification_algebra_decompostions}
    Let $\mathfrak{sl}(n+1, F) = \oplus_1^m \fg_i$ be an $(m,k)$-regular decomposition. Then
    \begin{enumerate}
        \item $(m,k) = (k+1,k)$ or $(m,k) = (k,k)$;
        \item Up to swapping positive and negative roots, re-numbering $\fg_i$'s and the action of $W(A_n)$ any $(k+1, k)$-regular decomposition is of the form
        \begin{equation*}
        \begin{aligned}
            \fg_1 &= \textnormal{span}_F \{
            E_{1, j} \mid 2 \le j \le n+1 \}, \\
            \fg_\ell &= \textnormal{span}_F \Bigg\{E_{i+1,j}, \ H_i \bigg| \sum_{t=1}^{\ell-2} \lambda_t < i \le \sum_{t=1}^{\ell-1} \lambda_t, \ 1 \le j \neq i+1 \le n+1  \Bigg\},
        \end{aligned}
    \end{equation*}
    where  $2 \le \ell \le k+1$ and $(\lambda_1, \dots, \lambda_k)$ is a $k$-partition of $n$.
    \item Up to the same equivalences, any $(k,k)$-regular decomposition is given by 
    \begin{equation*}
    \begin{aligned}
      \fg_1 &= \textnormal{span}_F \{ 
      E_{i+1,j}, H_i, X \mid 0 \le i \le \lambda_1, \ 1 \le j \neq i+1 \le n+1
      \}, \\
      \fg_\ell &= \textnormal{span}_F \Bigg\{
      E_{i+1,j}, H_i - X \bigg|
      \sum_{1}^{\ell - 1} \lambda_t < i 
      \le \sum_1^{\ell} \lambda_t, \ 0 \le j \neq i+1 \le n+1
      \Bigg\},
    \end{aligned}
    \end{equation*}
    where $2 \le \ell \le k$, \ $(\lambda_1, \dots, \lambda_k)$ is a $k$-partition of $n$ and 
    $$
    X \in \textnormal{span}_F \{H_{1}, \dots, H_{\lambda_1} \} \cup 
    \left\{ 
    H_p \mid 2 \le q \le k, \  \lambda_q > 1, \ 
    \sum_{t=1}^{q-1} \lambda_t < p \le \sum_{t=1}^{q} \lambda_t
    \right\}
    $$ 
    is an arbitrary element.
    \end{enumerate}
\end{maintheorem}

Combining \cref{mainthm:classification_algebra_decompostions} with the results of \cite{dokovic_check_hee} we get a complete description of all $(m \ge 2)$-regular decompositions of simple Lie algebras $\fg$ over an algebraically closed field of characteristic $0$.

\subsection*{Acknowledgments}
The author is thankful to R.\ Abedin, E. Karolinsky and A.\ Stolin for their helpful advice. \\
The work was supported by DFG -- SFB -- TRR 358/1 2023 -- 491392403.

\section{Preliminaries}
Let $F$ be an algebraically closed field of characteristic $0$ and $\fg$ be a simple Lie algebra over $F$.
Fix a Cartan subalgebra $\fh \subset \fg$. 
Let $\Delta$ be the set of non-zero roots of $\fg$ with respect to $\fh$ and 
\begin{equation*}
\fg = \fh \bigoplus_{\alpha \in \Delta} \fg_\alpha
\end{equation*}
be the corresponding root space decomposition.
We choose a set of simple roots $\pi = \{ \alpha_1, \dots, \alpha_n \}$ and denote by $\Delta_+$ the set of all positive roots.
Furthermore, we fix a basis $\{H_{\alpha_i}, E_{\pm\alpha} \mid 1 \le i \le n, \ \alpha \in \Delta_+\}$ of $\fg$ such that 
\begin{enumerate}
    \item $[E_\alpha, E_{-\alpha}] = H_\alpha = \sum_{i=1}^n c_i H_{\alpha_i}$  for 
    $\alpha = \sum_{i=1}^n c_i \alpha_i$ and
    \item $\kappa(E_\alpha, E_{-\alpha}) = 1$ for the Killing form $\kappa$ on $\fg$.
\end{enumerate} 
For all $1 \le i,j \le n$ we let
\begin{equation*}
    \langle \alpha_i, \alpha_j \rangle \coloneqq
    \alpha_i(H_{\alpha_j}) =
    \alpha_j(H_{\alpha_i})=
    \kf(H_{\alpha_i}, H_{\alpha_j}).
\end{equation*}

In this paper, our goal is to describe \emph{$m$-regular decompositions of $\fg$}, i.e.\ decompositions 
\begin{equation}%
\label{eq:regular_decomposition}
    \fg = \bigoplus_{i = 1}^m \fg_i, \ m \ge 2
\end{equation}
satisfying the following restrictions:
\begin{enumerate}
    \item all $\fg_i$ as well as $\fg_i \oplus \fg_j$ are Lie subalgebras of $\fg$;
    \item each $\fg_i$ has the form $ \fs_i \bigoplus_{\alpha \in \Delta_i} \fg_\alpha$ for some subspace $\fs_i \subseteq \fh$ and some subset $\Delta_i \subseteq \Delta$.
\end{enumerate}
Equivalently, all $\fg_i$ and $\fg_i \oplus \fg_j$ are regular subalgebras of $\fg$, which motivates the name.
When we want to emphasize that exactly $k$ summands in \cref{eq:regular_decomposition} have $\fs_i \neq 0$, we say that it is a \emph{regular decomposition of type $(m,k)$}.

A subset $S \subseteq \Delta$ is called \emph{closed}, if for any $\alpha, \beta \in S$, the inclusion $\alpha + \beta \in \Delta$ implies $\alpha + \beta \in S$.
Decompositions of $\fg$ into two regular subalgebras are closely related to partitions 
$$
\Delta = S_1 \sqcup S_2
$$
of the root system of $\fg$ into disjoint union of two closed subsets.

\begin{proposition}%
\label{prop:2_partititons}
    Given a partition $\Delta = S_1 \sqcup S_2$ into two closed subsets, we can find two (non-unique) subspaces $\fs_1, \fs_2 \subseteq \fh$ such that
    $$
    \fg = \left( \fs_1 \bigoplus_{\alpha \in S_1} \fg_\alpha \right) \oplus 
    \left( \fs_2 \bigoplus_{\alpha \in S_2} \fg_\alpha \right)
    $$ 
    is a regular decomposition.
    Conversely, by forgetting the Cartan part of a $2$-regular decomposition of $\fg$ we obtain a partition of $\Delta$ into two closed subsets.
\end{proposition}
\begin{proof}
    For a given partition $\Delta = S_1 \sqcup S_2$, $S_i \neq \emptyset$, we let 
    $$
    S_i^s \coloneqq \{ \alpha \in S_i \mid \alpha, -\alpha \in S_i \} = S_i \cap (-S_i)
    $$ 
    be the symmetric parts of $S_i$. 
    Consider the following two subalgebras of $\fg$:
    $$
    \fg_i^s \coloneqq \langle E_\alpha \mid \alpha \in S_i^s \rangle.
    $$
    By \cite[Lemma 2]{dokovic_check_hee}
    the sets $S_1^s$ and $S_2^s$ are strongly orthogonal. In particular, this means that 
    $$\fg_1^s \cap \fg_2^s = 0.$$
    

    Let $S_i^a \coloneqq S_i \setminus S_i^s$ be the anti-symmetric parts of $S_i$. Consider the following mutually disjoint regular subalgebras of $\fg$:
    $$
    \fg_i' \coloneqq \fg_i^s \oplus \textnormal{span}_F 
    \{ E_\alpha \mid \alpha \in S_i^a\}.
    $$
    Their sum is $\fg$ with a missing Cartan part $\fh'$.
    Indeed, if the sum is the whole $\fg$, then the sets $S_i^s$ are not mutually orthogonal or one of $S_i$ is empty. 
    This implies that
    $$
    \fg = (\fh' \oplus \fg'_1) \oplus \fg'_2
    $$
    is a desired regular decomposition of $\fg$.
    
    Note that the missing part $\fh'$ can be chosen and distributed among $\fg_1'$ and $\fg_2'$ completely arbitrary.
    The converse direction is obvious.
\end{proof}

The classification, up to the action of the Weyl group $W(\Delta)$, of all closed subsets $S \subset \Delta$, such that $\Delta \setminus S$ is also closed, was obtained in \cite{dokovic_check_hee}. 
Combining this result with
\cref{prop:2_partititons} and its proof, we get a description of all $2$-regular decompositions of $\fg$ up to the action of $W(\Delta)$ and arbitrary distribution of a missing Cartan part $\fh'$ into $\fg_1^s$ and $\fg_2^s$.
Consequently, in this paper we concentrate on $m$-regular decompositions with $m \ge 3$.

By analogy with $m$-regular decompositions of $\fg$ we define an \emph{$m$-regular partition of the finite root system $\Delta$} as a partition of the form
$$
    \Delta = \bigsqcup_{i=1}^m \Delta_i
$$
with the property that all $\Delta_i$ and $\Delta_i \sqcup \Delta_j$ are closed.

It is tempting to assume a general version of \cref{prop:2_partititons}, namely, that any $m$-regular partition of $\Delta$ gives rise to an $m$-regular decomposition of $\fg$.
This statement turns out to be true.
We prove it in the following chapters and use it to reduce the classification problem of decompositions of $\fg$ to the question of regular partitions of $\Delta$.

\section{Regular partitions of finite root system}
Let $\fg$ be a simple Lie algebra over an algebraically closed field $F$ of 0 characteristic.
As before, the irreducible root system of $\fg$ is denoted by $\Delta$.
We let $n$ be the rank of $\Delta$.

\subsection{Graph $G$ associated with an $m$-regular partition}
Consider an arbitrary $m$-regular partition
\begin{equation}%
\label{eq:fixed_m_partition}
    \Delta = \bigsqcup_{i=1}^m \Delta_i, \ m \ge 3.
\end{equation}
Given an (integral) basis $\{ \beta_i \}_{1}^n \subseteq \Delta$ of the root system, we can associate the partition
\cref{eq:fixed_m_partition} with 
a labeled multigraph $G$ on $m$ vertices with loops:

\begin{itemize}
    \item Vertices of the graph are integers $1, 2, \dots, m$;
    \item There is an edge between vertices $1 \le i \neq j \le m$
    when we can find $1 \le \ell \le n$ such that $\beta_\ell \in \Delta_i$ and $-\beta_{\ell} \in \Delta_j$. We label that edge with $\beta_\ell$ and equip its ends with $+$ and $-$ signs, to remember how the roots are distributed (see \cref{ex:easy_example_graph});
    \item We endow the vertex $i$ with a loop $\beta_\ell$ if $\pm \beta_\ell \in \Delta_i$.
\end{itemize}
To avoid ambiguity in the sequel, we never use the word edge to refer to a loop.
We sometimes omit certain labels of the graph when they are not important.

Note that the graph $G$ does depend on the choice of a basis $\{ \beta_i \}$.
Later, by selecting specific bases, we obtain significant constraints on the associated graph.

\begin{example}%
\label{ex:easy_example_graph}
Consider the $3$-partition of $A_3$
\begin{equation*}
\begin{aligned}
    &\{-\beta, -\beta - \gamma, \alpha \} \\
    & \{-\gamma, \beta, \alpha + \beta \} \\
    & \{\gamma, \beta + \gamma, \alpha + \beta + \gamma, -\alpha - \beta - \gamma, -\alpha, -\alpha - \beta \} 
\end{aligned}
\end{equation*}
The graph of this partition with respect to the basis $\{\alpha, \alpha+\beta, \alpha+\beta+\gamma \}$ is presented in \cref{fig:3_partition}.
\end{example}
\begin{figure}[H]
    \centering
    \includegraphics[scale = 0.7]{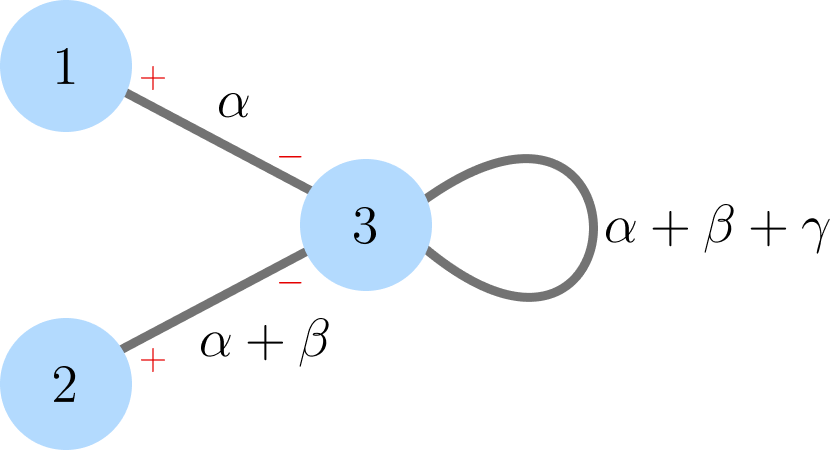}
    \caption{Graph of a 3-partition.}
    \label{fig:3_partition}
\end{figure}
\noindent 
The properties of such a graph depend on the type of a root system $\Delta$.
For this reason, we break our considerations into several subsections.

\subsubsection{Type $A_n$}
Let $\pi = \{\alpha_1, \alpha_2, \dots, \alpha_n \}$ be a set of symple roots of $\Delta$.
We define the sums
\begin{align*}
    \beta_i \coloneqq \alpha_1 + \alpha_2 + \dots + \alpha_i, \ 1 \le i \le n.
\end{align*}
It is not hard to see, that the set $\pi$ can be re-ordered in such a way, that all $\beta_i$ as well as all their differences $\beta_i - \beta_j$, $1 \le i \neq j \le n$, are roots.
More precisely, there are two such orderings.
If simple roots of $A_n$ correspond to vertices of the Dynkin diagram as shown in \cref{fig:An_string},
then the desired orderings are
$\{\alpha_1, \dots, \alpha_n\}$ and $\{\alpha_n, \dots, \alpha_1\}$.
Without loss of generality, we choose the first one.
\begin{figure}[H]
    \centering
    \includegraphics[scale = 0.7]{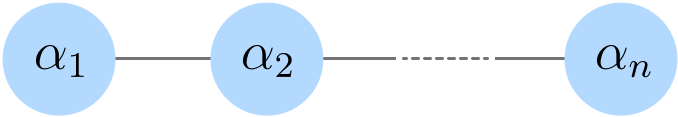}
    \caption{Dynkin diagram for $A_n$.}
    \label{fig:An_string}
\end{figure}

The set $\{\beta_i \}_1^n$ is a basis for $\Delta$. 
Consider the graph $G$ of an $m$-regular partition of $\Delta$ with respect to this basis

\paragraph{Property A1.}
The graph $G$ can have at most one vertex with loops.

\begin{proof}
Assume the opposite. Then we can find two indices $1 \le i \neq j \le n$, such that $\pm \beta_i \in \Delta_1$ and $\pm \beta_j \in \Delta_2$. 
Since we started with a regular splitting, we must have $\pm (\beta_i - \beta_j) \in \Delta_1 \sqcup \Delta_2$.
If $\beta_i - \beta_j \in \Delta_1$, then $-\beta_j \in \Delta_1$ which is impossible because by our assumption $-\beta_j \in \Delta_2$.
Similarly, placing $\beta_i - \beta_j$ into $\Delta_2$ leads to the contradiction $\beta_i \in \Delta_2$.
In other words, there cannot be two vertices with loops.
\end{proof}

\paragraph{Property A2.}
Each vertex $i$ either has a loop $\beta_\ell$ or there is an edge
\begin{tikzcd}
i \arrow[r, "\beta_\ell", no head] & j
\end{tikzcd}
for some $j \neq i$.

\begin{proof}
    Assume a vertex $i$ has no loops and is not connected to any other vertex.
    By definition of $G$ this means that all $\pm \beta_\ell$, $1 \le \ell \le n$, are contained in $\sqcup_{j \neq i} \Delta_j$.
    Since any root can be written as a linear (integer) combination of certain $\beta_\ell$, we get the equality $\Delta = \sqcup_{j \neq i} \Delta_j$,
    raising the contradiction $\Delta_i = \emptyset$. 
\end{proof}

\paragraph{Property A3.}
Two edges incident to a single vertex have the same signs at that vertex.

\begin{proof}
    Assume our graph contains a subgraph of the form
    \begin{figure}[H]
    \centering
    \includegraphics[scale = 0.7]{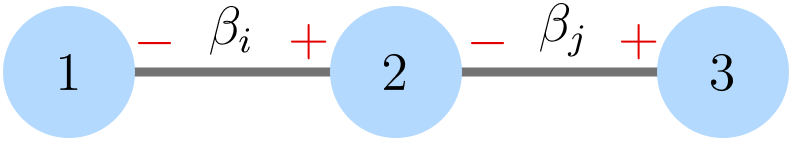}
\end{figure}
\noindent Then we know that $\beta_j - \beta_i \in \Delta_1 \sqcup \Delta_3$.
If $\beta_j - \beta_i \in \Delta_1$, then $$\Delta_3 \ni \beta_j = (\beta_j - \beta_i) + \beta_i \in \Delta_1 \sqcup \Delta_2 $$
which is absurd.
Similarly, if $\beta_j - \beta_i \in \Delta_3$, then $-\beta_i$ must simultaneously lie in $\Delta_1$ and $\Delta_2 \sqcup \Delta_3$.
\end{proof}

\paragraph{Property A4.}
Any two edges in $G$ share a common vertex.

\begin{proof}
If we have two edges
\begin{figure}[H]
    \centering
    \includegraphics[scale = 0.7]{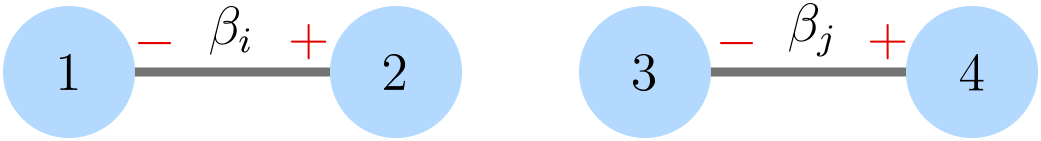}
\end{figure}
\noindent then, by identifying vertices $2$ and $3$ (equivalently, replacing $\Delta_2 \sqcup \Delta_3$ with a single closed subset $\Delta_{23}$) and repeating word-by-word the proof of Property A3, we get a contradiction.
\end{proof}

\paragraph{Property A5.}
A loop in $G$ can occur only in the vertex common to all edges.

\begin{proof}
Assume $G$ has a subgraph of the form
\begin{figure}[H]
    \centering
    \includegraphics[scale = 0.7]{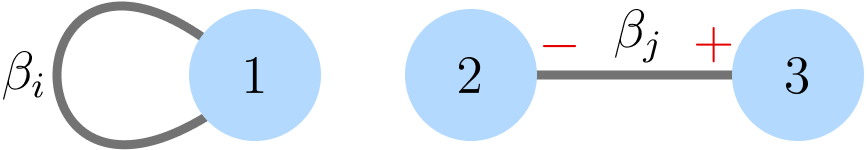}
\end{figure}
\noindent We then have the containment
$\beta_i - \beta_j \in \Delta_1 \sqcup \Delta_2$.
The first containment $\beta_i - \beta_j \in \Delta_1$ leads to the contradiction $- \beta_j \in \Delta_1$.
In the case $\beta_i - \beta_j \in \Delta_2$, we get $\beta_i \in \Delta_2 \sqcup \Delta_3$ which is again impossible.
Consequently, if $G$ contains a loop, then it must be at a vertex common to all the edges.
\end{proof}

Combining all the preceding results, we see that the graph $G$, associated to a regular partition \cref{eq:fixed_m_partition}, up to re-numbering of sets $\Delta_i$ and swapping the sign labels, is necessarily of the form
\begin{figure}[H]
    \centering
    \includegraphics[scale = 0.7]{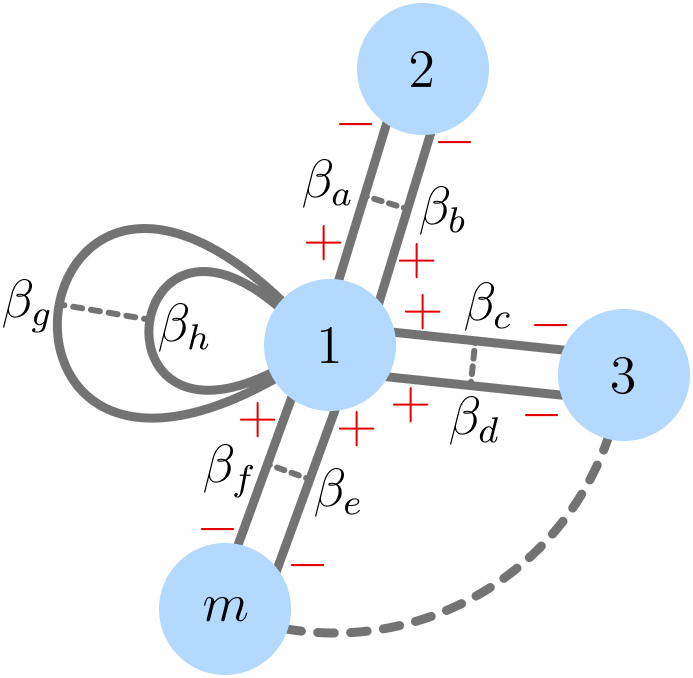}
    \caption{Graph $G$ of an arbitrary $m$-regular partition with $m \ge 3$.}
    \label{fig:graph_partition_A}
\end{figure}
\noindent The total number of loops and edges is clearly equal to the rank $n$ of the root system and $m \le n+1$.

\begin{theorem}%
\label{thm:graph_partition_A}
A labeled multigraph $G$ of the form presented in \cref{fig:graph_partition_A} determines uniquely an $m$-regular partition of $A_n$ with $m \ge 3$.
\end{theorem}
\begin{proof}
The graph $G$ completely describes where the roots $\pm \beta_i$, $1 \le i \le n$, live. 
Therefore, it is enough to show that the remaining roots can be distributed among $\Delta_i$'s in a unique way.

Take a root $\alpha \in \Delta \setminus \{ \beta_i \}_1^n$ and write it as a difference $\beta_j - \beta_i$.
We then have the following cases:

\begin{enumerate}
    \item $\beta_i$ is a loop and hence the roots $\beta_j, -\beta_i, \alpha \in \Delta_1$;
    \item $\beta_i$ is not a loop and hence we can assume without loss of generality, that $-\beta_i \in \Delta_2$. Suppose $\beta_j$ is a loop. Then $\alpha = \beta_j - \beta_i \in \Delta_1 \sqcup \Delta_2$.
    The containment $\alpha \in \Delta_1$ leads to the contradiction $-\beta_i \in \Delta_1$. Consequently, $\alpha \in \Delta_2$;
    \item $-\beta_i \in \Delta_2$, $-\beta_j \in \Delta_3$ and $\alpha \in \Delta_1 \sqcup \Delta_2$. 
    Again, $\alpha \in \Delta_1$ leads to the false statement
    $$
    \Delta_1 \sqcup \Delta_3 \ni (\beta_j - \beta_i) - \beta_j = -\beta_i \in \Delta_2.
    $$
    Therefore, $\alpha \in \Delta_2$;
    \item The final case is when both $-\beta_i, -\beta_j \in \Delta_2$ and $\alpha \in \Delta_1 \sqcup \Delta_2$. Since $m \ge 3$, there must be another index $1 \le k \le n$ such that $-\beta_k \in \Delta_3$.
    By 3.\ above we must have $\beta_k - \beta_i \in \Delta_2$ and $\beta_j - \beta_k \in \Delta_3$. 
    If $\alpha \in \Delta_1$, then
    $$
    \Delta_2 \sqcup \Delta_3 \ni (\beta_k - \beta_i) + (\beta_j - \beta_k) = \beta_j - \beta_i \in \Delta_1,
    $$
    which is impossible.
    In other words, we have $\alpha \in \Delta_2$.
\end{enumerate}
\noindent This argument gives a unique partition of $A_n$ into $m$ subsets.
These subsets as well as their unions are closed by construction, providing us with the desired $m$-regular partition.
\end{proof}

It now follows, that the graph of an $(n+1)$-regular partition must have one of the forms presented in \cref{fig:maximal_partition_A}.
\begin{figure}[h]
    \centering
    \includegraphics[scale = 0.7]{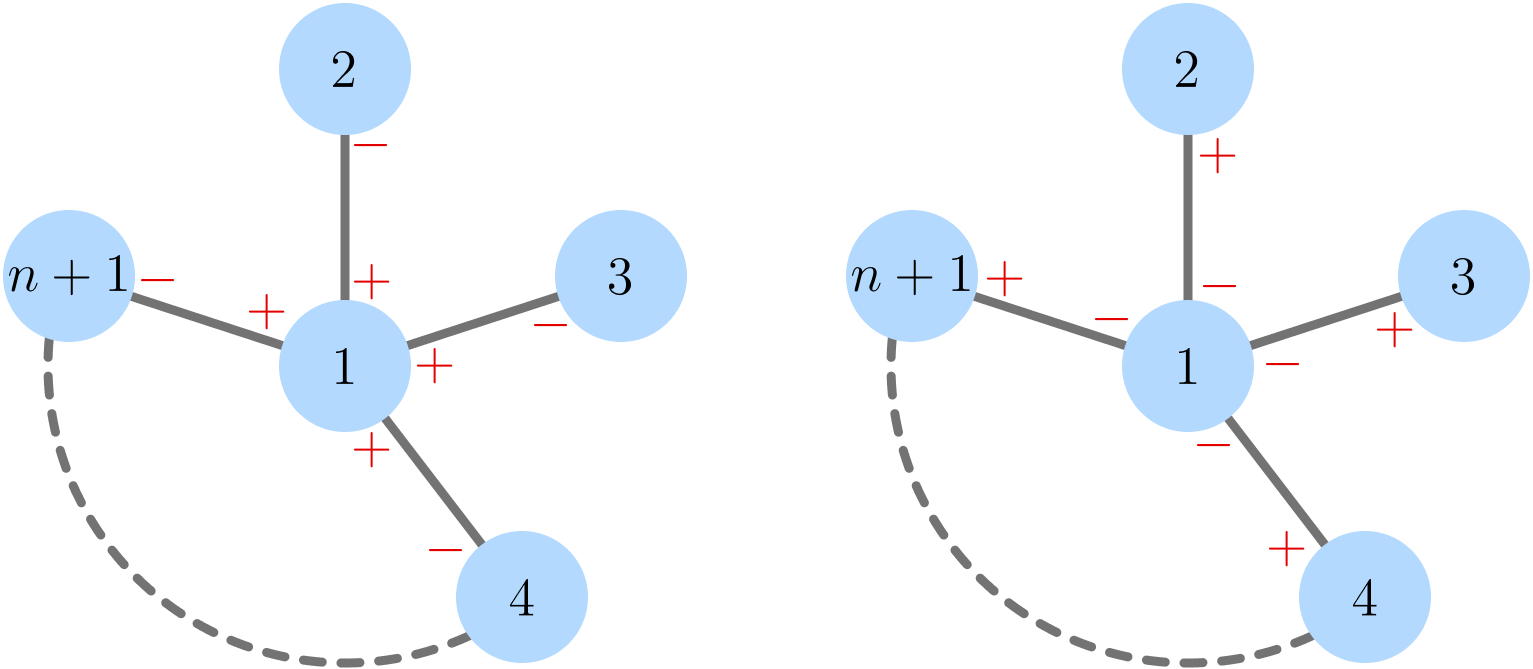}
    \caption{Graphs of two maximal regular partitions.}
    \label{fig:maximal_partition_A}
\end{figure}
\noindent Explicitly, up to re-numbering of $\Delta_i$'s these partitions are given by
\begin{equation}%
\label{eq:finest_An}
\renewcommand{\arraystretch}{1.2}
\begin{tabular}{ l c l }
 $\Delta_0 = \{\beta_i \mid 1 \le i \le n\}$ & and & $\Delta_j = \{-\beta_j, -\beta_j + \beta_i \mid 1 \le i \neq j \le n \},$ \\ 
 $\Delta_0 = \{-\beta_i \mid 1 \le i \le n\}$ & and &
 $\Delta_j = \{\beta_j, \beta_j - \beta_i \mid 1 \le i \neq j \le n \},$
\end{tabular}
\end{equation}
where  $1 \le j \le n$. 
Setting $\beta_0 \coloneqq 0$, we can write $$\Delta_i = \pm \{-\beta_i + \beta_j \mid 0 \le i \neq j \le n \}, \ \ 0 \le i \le n.$$

\begin{remark}%
\label{rem:change_of_order}
If we let
$\overleftarrow{\beta_k} \coloneqq \alpha_{n} + \dots + \alpha_{n-k+1}$ and $\overleftarrow{\beta_0} \coloneqq 0$, then
$$
\beta_j - \beta_i = -\overleftarrow{\beta_{n-j}} + \overleftarrow{\beta_{n-i}}
$$
for all $1 \le i \neq j \le n$.
Therefore, choosing the opposite sign-labels is equivalent to choosing the opposite order of the roots.
\end{remark}

\begin{remark}
    The finest partitions \cref{eq:finest_An} have another simple interpretation. 
    Viewing $\mathfrak{sl}(n+1,F)$ as the Lie algebra of traceless $(n+1)\times (n+1)$-matrices, the set $\Delta_k$ from the first partition corresponds to the subalgebra generated by root vectors $E_{k+1, j}$, $1 \le j \neq k+1 \le n+1$, i.e.\ root vectors ``in row $k+1$''.
    Similarly, the set $\Delta_k$ from the second finest partition corresponds to the root vectors ``in column $k+1$''.
    \begin{figure}[h]
    \centering
    \includegraphics[scale = 0.7]{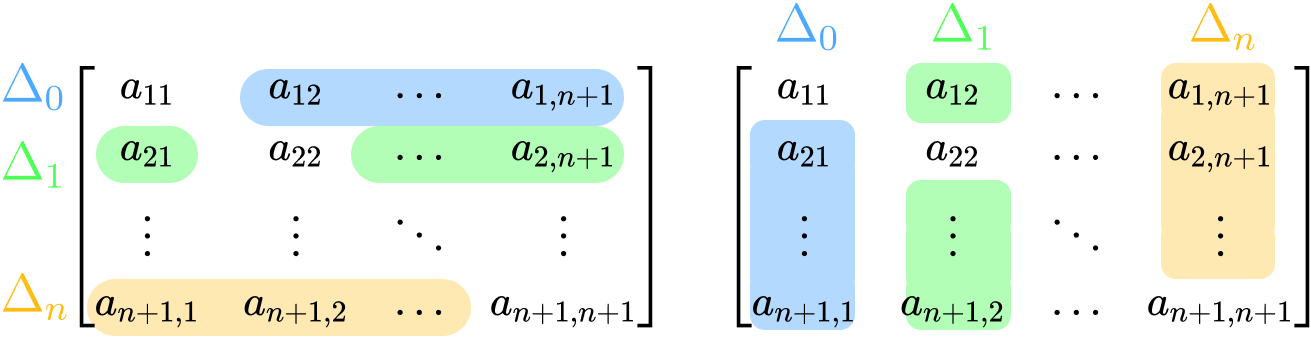}
\end{figure}
For that reason, partitions \cref{eq:finest_An} are called \emph{row} and \emph{column partitions} respectively.
\end{remark}

Replacing two elements $\Delta_i$ and $\Delta_j$ of an $m$-regular partition with a single subset $\Delta_i \sqcup \Delta_j$ is equivalent to identifying vertices $i$ and $j$ of the corresponding graph $G$. 
When one of the indices is equal to $1$, this operation creates loops. 
In other cases it produces multiple edges.
Since any graph of the form \cref{fig:graph_partition_A} can be obtained from graphs \cref{fig:maximal_partition_A} by performing such identifications, we get the following result.
\begin{corollary}%
\label{cor:any_decomposition_extends}
    Any regular partition of $A_n$ into $m \ge 3$ parts can be obtained from one of the finest partitions \cref{eq:finest_An} by replacing several $\Delta_i$'s with their disjoint unions.
\end{corollary}
Therefore, up to re-numbering of $\Delta_i$'s and swapping the sign labels, there are 
$$
\sum_{k=3}^{n+1} {n+1 \brace k}
$$
$m$-regular partitions with $m \ge 3$. Here ${n \brace k}$ is the Stirling number of the second kind.
The next lemma allows to take into account the action of the corresponding Weyl group $W(\Delta)$.

\begin{lemma}%
\label{lem:Weyl_group_action}
  Let $P \coloneqq \{\Delta_0, \Delta_1, \dots,  \Delta_n\}$ be the set of elements of one of the finest partitions described in \cref{eq:finest_An}.
  The Weyl group $W(\Delta)$ acts on $P$ as the symmetric group $S_{n+1}$.
\end{lemma}
\begin{proof}
    Consider the first partition in \cref{eq:finest_An}.
    Let us examine the action of the simple reflections $s_i \coloneqq s_{\alpha_i}$ on the elements $\beta_1 \in \Delta_0$ and $-\beta_k \in \Delta_k$, $2 \le k \le n$.
    The only simple reflections impacting $-\beta_k$ are $s_{1}, s_{k}$ and $s_{k+1}$ with the actions given by
    $$
    \begin{cases}
      s_1(-\beta_k) = -\beta_k + \beta_1 \in \Delta_k, \\
      s_k(-\beta_k) = -\beta_{k-1} \in \Delta_{k-1}, \\
      s_{k+1}(-\beta_k) = -\beta_{k+1} \in \Delta_{k+1}.
    \end{cases}
    $$
    Consequently, the simple reflection $s_k$, $2 \le k \le n$, swaps elements $\Delta_{k-1}$ and $\Delta_{k}$ and $s_1$ interchanges $\Delta_0$ with $\Delta_1$, leaving all other elements of the partition untouched. 
    The proof is now complete, because transpositions $(j, j+1)$ with $1 \le j \le n$ generate $S_{n+1}$.
\end{proof}
\begin{remark}
    Another way to get the same result is to view the Weyl group $W(\Delta)$ as the group $N_{G}(T)/T$, where $G = GL(n, F)$ and $T \subset G$ is the subgroup of diagonal matrices.
    The statement now follows from the observation that $N_{G}(T)/T$ contains all the permutation matrices.
\end{remark}

Combining \cref{cor:any_decomposition_extends} with
\cref{lem:Weyl_group_action} we see that any $m$-regular partition of $A_n$ is completely determined by a partition $\lambda = (\lambda_1, \dots, \lambda_m)$ of the number $(n+1)$.

\begin{proposition}%
\label{prop:m_regular_An}
    Up to swapping positive and negative roots and the action of the Weyl group $W(\Delta)$,
    any $m$-regular partition $\Delta = \sqcup_1^m \Delta_i$ 
    with $m \ge 3$ has the form
    \begin{equation}
        \begin{aligned}
            \Delta_1 &= \{ -\beta_i + \beta_j \mid 0 \le i < \lambda_1, \ 0 \le j \neq i \le n  \}, \\
            \Delta_2 &= \{ -\beta_i + \beta_j \mid \lambda_1 \le i < \lambda_1 + \lambda_2, \ 0 \le j \neq i \le n  \}, \\
            &\vdots \\
            \Delta_m &= \{ -\beta_i + \beta_j \mid \lambda_1 + \dots + \lambda_{m-1} \le i < n+1, \ 0 \le j \neq i \le n  \},
        \end{aligned}
    \end{equation}
    where $\lambda = (\lambda_1, \dots, \lambda_m)$ is an $m$-partition of $(n+1)$. 
\end{proposition}

This completes the classification in the first case. 
Note that the number of partitions $p(n+1)$ is known to grow exponentially as $\exp(\Theta(\sqrt{n}))$, where $f(n) = \Theta(g(n))$ means that $f$ is bounded both below and above by $g$; see \cite{andrew_theory_of_partitions}.

\subsubsection{Type $B_n$}
Assume $\Delta$ is of type $B_n$, $n \ge 2$.
Fix a set of simple roots $\pi = \{\alpha_1, \dots, \alpha_n \}$ ordered in accordance with the Dynkin diagram in \cref{fig:Bn_string}.
\begin{figure}[H]
    \centering
    \includegraphics[scale = 0.7]{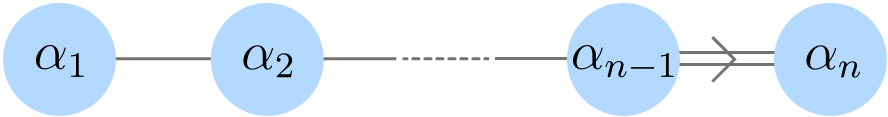}
    \caption{Dynkin diagram for $B_n$.}
    \label{fig:Bn_string}
\end{figure}
\noindent Again, for $1 \le i \le n$ we put $\beta_i \coloneqq \alpha_1 + \dots + \alpha_i$ and $\beta_0 \coloneqq 0$.
Then the set $$S \coloneqq \{ \beta_i - \beta_j \mid 0 \le i \neq j \le n \}$$
is a subset of $\Delta$.
Observe that $S$ is not a closed subset of $\Delta$: for example, the maximal root $\beta_n + (\beta_n - \beta_1)$ is not in $S$.
However, the roots in $S$ satisfy the same linear relations as the corresponding roots in $A_n$.
More precisely, a linear combination of roots $\beta_i - \beta_j$ is again a root in $S$ if and only if the same linear combination is a root in $A_n$.


The existence of such a subset guarantees, that the graph $G$
of an $(m \ge 3)$-regular partition of $\Delta$ with respect to the basis $\{ \beta_i \}_1^n$ has all the properties A1 -- A5 above.
In particular, up to re-numbering the sets $\Delta_i$ and swapping positive and negative roots, it must be of the form \cref{fig:graph_partition_A}.
However, the presence of additional roots impose further restrictions on $G$.

\paragraph{Property BC1.} The graph $G$, associated to a regular partition of $B_n$, cannot contain subgraphs of the form
\begin{figure}[H]
    \centering
    \includegraphics[scale = 0.7]{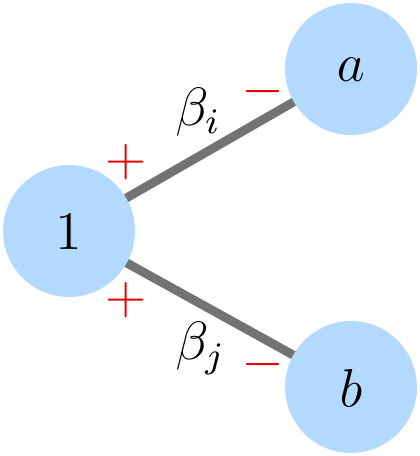}
\end{figure}
\noindent with $1 < a \neq b \le m$ and $1 \le i \neq j \le n$.
\begin{proof}
    Assume the opposite. We first consider the case $i,j < n$.
    Repeating the argument from the proof of \cref{thm:graph_partition_A} we obtain the contaiments
    $$
    \beta_n - \beta_i \in \Delta_a \ \text{ and } \ \beta_n - \beta_j \in \Delta_b.
    $$
    It follows that $2\beta_n - \beta_i \in \Delta_1 \sqcup \Delta_a$ and $2\beta_n - \beta_j \in \Delta_1 \sqcup \Delta_b$.
    Let us show, that the latter two roots cannot lie in $\Delta_1$.
    Suppose $2 \beta_n - \beta_i \in \Delta_1$ and $-\beta_n \not \in \Delta_a$, then
    $$
    \Delta_a \not \ni (2\beta_n - \beta_i) - \beta_n = \beta_n - \beta_i \in \Delta_a
    $$
    which is impossible.
    If $-\beta_n \in \Delta_a$, then by previous argument $2\beta_n - \beta_j \in \Delta_b$ and hence
    $$
    \Delta_a \sqcup \Delta_b \ni (2 \beta_n - \beta_j) + (\beta_j - \beta_i) = 2 \beta_n - \beta_i \in \Delta_1.
    $$
    Therefore, $2\beta_n - \beta_i \in \Delta_a$ and similarly $2\beta_n - \beta_j \in \Delta_b$.
    
    Furthermore, the root $2\beta_n - \beta_j - \beta_i$ must lie in $\Delta_a \sqcup \Delta_b$, which is impossible.
    Indeed, if $2\beta_n - \beta_j - \beta_i \in \Delta_a$, then
    $$
    \Delta_b \ni 2 \beta_n - \beta_j = (2 \beta_n - \beta_j - \beta_i) + \beta_i \in \Delta_1 \sqcup \Delta_a.
    $$
    Similarly with the second containment. 
    
    Finally, assume $j = n$.
    Then, again, $2\beta_n - \beta_i \in \Delta_a$ and $\beta_i -  \beta_n \in \Delta_b$.
    Implying
    $$\Delta_1 \ni \beta_n = (2 \beta_n - \beta_i) + (\beta_i - \beta_n) \in \Delta_a \sqcup \Delta_b.$$
\end{proof}
Since any regular partition of $B_n$  into $m \ge 3$ parts must contain a subgraph of the form mentioned in Property B1, we have the following statement. 
\begin{proposition}
There are no regular partitions of $B_n$ into $m \ge 3$ parts.    
\end{proposition}

\subsubsection{Type $C_n$}
Since $C_2$ is isomorphic to $B_2$ we let $n \ge 3$.
Similarly to the previous case, by
ordering the simple roots of $C_n$ in 
accordance with \cref{fig:Cn_string} and defining $\beta_i$ as before, we get a subset $S \coloneqq \{ \beta_i - \beta_j \mid 0 \le i \neq j \le n \} \subset \Delta$ that behaves like $A_n$.
\begin{figure}[H]
    \centering
    \includegraphics[scale = 0.7]{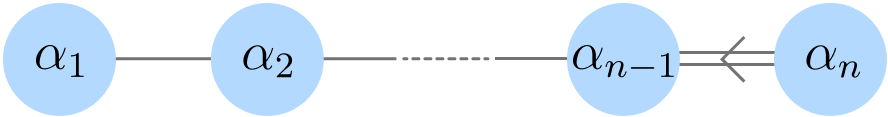}
    \caption{Dynkin diagram for $C_n$.}
    \label{fig:Cn_string}
\end{figure}
\noindent Therefore, up to interchanging the positive and negative roots and re-numbering $\Delta_i$, the graph of an $(m \ge 3)$-regular partition of $C_n$ with respect to $\{\beta_i \}_1^n$ is again of the form
\cref{fig:graph_partition_A}.

We now prove that $G$ has the property BC1 and, consequently, there are no regular partitions of $C_n$ into $3$ or more parts.

\begin{proposition}%
\label{prop:regular_Bn}
    The are no regular partitions of $C_n$ into $m \ge 3$ parts.
\end{proposition}
\begin{proof}
Assume there is such a partition.
    First of all, we observe that
    $\pm (\beta_n + \beta_{n-1})$ are again roots.
    This implies that there must exist a unique $1 \le k \le m$ such that both roots $-\beta_n$ and $-\beta_{n-1}$ lie in $\Delta_k$.
    Otherwise, the root $-\beta_n - \beta_{n-1}$ cannot be placed in any of the parts.
    In other words, the edges $\beta_n$ and $\beta_{n-1}$ in $G$ cannot be separated.
    Now we can repeat the proof of the property BC1 for $B_n$ above with the roots
    $$(\beta_n + \beta_{n-1}) - \beta_i, \  (\beta_n + \beta_{n-1}) - \beta_j \ \text{ and } \ 
    (\beta_n + \beta_{n-1}) - \beta_i - \beta_j$$ 
    instead of $2\beta_n - \beta_i$, $2\beta_n - \beta_j$ and $2\beta_n - \beta_j - \beta_i$ respectively.
\end{proof}

\subsubsection{Type $D_n$}
For $\Delta$ of type $D_n$, $n \ge 4$, we arrange the simple roots $\pi = \{ \alpha_1, \dots, \alpha_n \}$ as depicted in \cref{fig:Dn_string}.
Let $R_1$ and $R_2$ be the subsystems of $\Delta$ spanned by the roots $\pi \setminus \{\alpha_n\}$ and $ \pi \setminus \{ \alpha_{n-1} \}$ respectively.
It is clear that both $R_1$ and $R_2$ are irreducible of type $A_{n-1}$.
\begin{figure}[H]
    \centering
    \includegraphics[scale = 0.7]{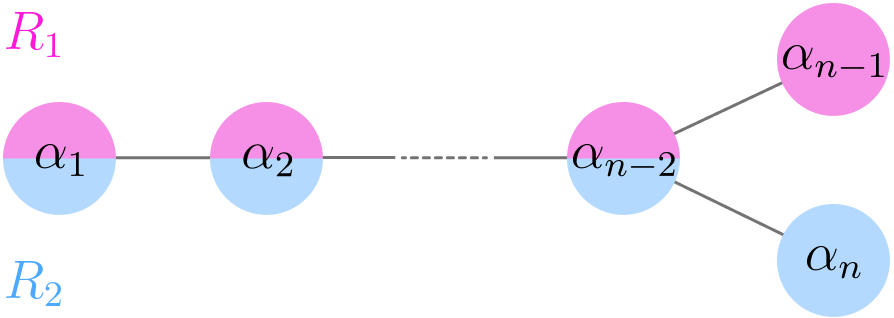}
    \caption{Subsystems $R_1, R_2 \subset \Delta$ of type $A_{n-1}$.}
    \label{fig:Dn_string}
\end{figure}
\noindent Any regular partition of $\Delta$ restricts in an obvious way to regular partitions of subsystems $R_1$ and $R_2$.
We now consider the corresponding graph $G$ with respect to following basis of $\Delta$:
\begin{align*}
    \gamma_i &\coloneqq \beta_i \coloneqq \alpha_1 + \alpha_2 + \dots + \alpha_{i}, \ 1 \le i \le n-1; \\
    \gamma_n &\coloneqq \alpha_1 + \dots + \alpha_{n-2} + \alpha_n.
\end{align*}
Observe that the graphs $G_1$ and $G_2$  of the induced regular partitions of $R_1$ and $R_2$ respectively with respect to bases 
$$\{\gamma_i \mid 1 \le i \le n-1\} \ \text{ and } \ \{ \gamma_i \mid 1 \le i \neq n-1 \le n \}$$ are subgraphs of $G$.
Furthermore, we can understand $G$ by understanding its subgraphs $G_1, G_2$ and how they can be glued together.

\begin{lemma}%
\label{lem:three_vertices_each_graph}
If $G$ has $m \ge 3$ vertices, then so do $G_1$ and $G_2$.
\end{lemma}
\begin{proof}
    By construction $m -1 \le |G_i| \le m$ for $1 \le i \le 2$. 
    Identifying together elements of the partition, we can restrict ourselves to the case $m = 3$.  
    Consequently, to prove the statement it is then enough to show, that there are no partitions with the following graphs
    \begin{figure}[H]
    \centering
    \includegraphics[scale = 0.7]{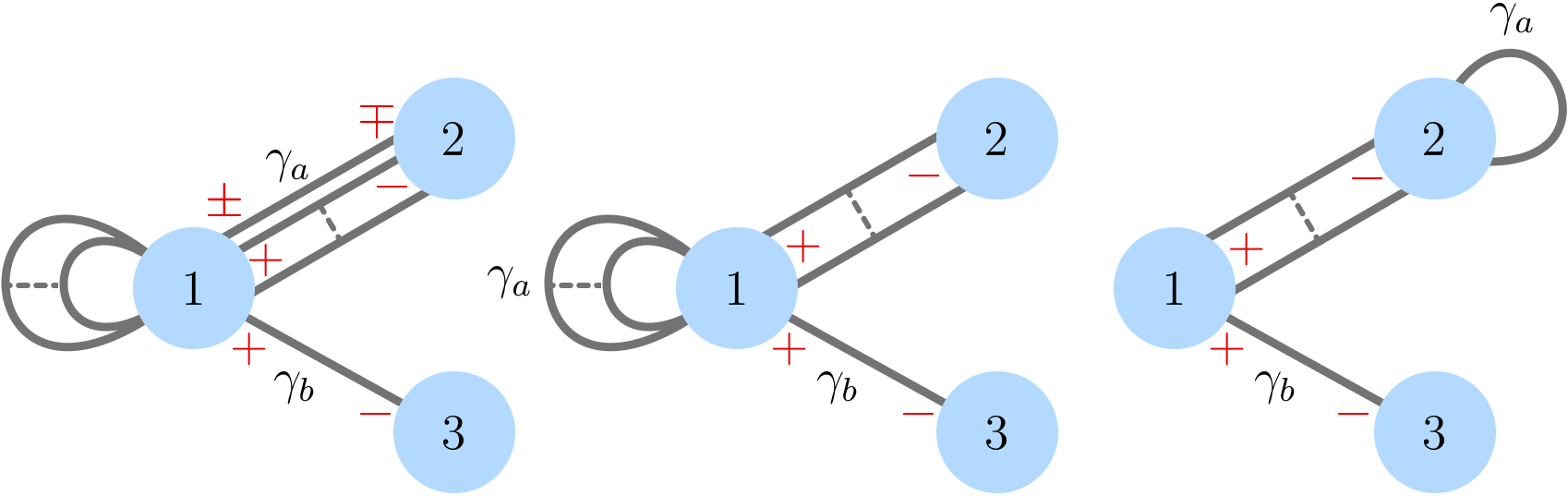}
\end{figure}
\noindent where $\{a,b \} = \{n,n-1\}$.
These are the only graphs on $3$ vertices, up to swapping the sign-labels, with the property that both $G_1$ and $G_2$ are of type $A_{n-1}$ and at least one of them contains $2$ vertices.

In all cases we have the contaiments
\begin{align*}
  &\pm (-\gamma_a + \gamma_i), \ \gamma_a + \gamma_b - \gamma_i, \ \gamma_a + \gamma_b - \gamma_i - \gamma_j \in \Delta_1 \sqcup \Delta_2,   \\
  &-\gamma_b + \gamma_i \in \Delta_3
\end{align*}
for all $1 \le i \neq j \le n-2$.
Therefore, $-\gamma_a - \gamma_b + \gamma_i \in \Delta_3$.
Otherwise 
$$\Delta_3 \ni -\gamma_b = (-\gamma_a - \gamma_b + \gamma_i) + (\gamma_a - \gamma_i) \in \Delta_1 \sqcup \Delta_2. $$
Moreover, since 
$$
\Delta_3 \ni -\gamma_a - \gamma_b + \gamma_i = (-\gamma_a - \gamma_b + \gamma_i + \gamma_j) - \gamma_j 
$$
and $-\gamma_j \in \Delta_1 \sqcup \Delta_2$, we must have
$-\gamma_a - \gamma_b + \gamma_i + \gamma_j \in \Delta_3$.
For this reason, the relation
$$
\Delta_1 \ni \gamma_i = (-\gamma_a - \gamma_b + \gamma_i + \gamma_j) + (\gamma_a + \gamma_b - \gamma_j)
$$
implies the containment $\gamma_a + \gamma_b - \gamma_j \in \Delta_1$, which, in its turn, leads to
\begin{equation}%
\label{eq:a_i_Delta_1}
    \gamma_a - \gamma_j\in \Delta_1
\end{equation}for all $1 \le j \le n-2$.
Moreover, if 
$\gamma_a + \gamma_b - \gamma_i - \gamma_j \in \Delta_2$, then
$$
\Delta_2 \sqcup \Delta_3 \ni (\gamma_a + \gamma_b - \gamma_i - \gamma_j) + (-\gamma_b + \gamma_j) = \gamma_a - \gamma_i \in \Delta_1.
$$
Therefore, we have 
$\gamma_a + \gamma_b - \gamma_i - \gamma_j \in \Delta_1$ for all
$1 \le i \neq j \le n-2$. 
Because of the equality
$$(\gamma_a + \gamma_b - \gamma_i - \gamma_j) + (-\gamma_a -\gamma_b + \gamma_j ) = -\gamma_i$$ 
all $\gamma_i$ must be loops. 
Combining this observation with \cref{eq:a_i_Delta_1} we see that $\gamma_a \in \Delta_1$ and $-\gamma_a \in \Delta_2$. 
This is already impossible in cases two and three. 
In the first case we get a further contradiction
$\Delta_1 \ni (-\gamma_a - \gamma_b + \gamma_i) + (\gamma_b - \gamma_i) = -\gamma_a \in \Delta_2 $.
\end{proof}

The statement of \cref{lem:three_vertices_each_graph} can also be reformulated in the following way. 
The graph of an $m$-regular partition of $D_n$, $m \ge 3$, never contains a single (isolated) edge $\gamma_n$ or $\gamma_{n-1}$.
Having this result at hand, we can show that such partitions are impossible.

\begin{proposition}%
\label{prop:D_n_partition}
    There are no regular partitions of $D_n$ with $m \ge 3$ parts.
\end{proposition}
\begin{proof}
    Let us assume the opposite.
    By identifying some $\Delta_i$'s we can assume $m = 3$.
    By \cref{lem:three_vertices_each_graph} the associated graph $G$, up to switching the sign labels and re-numbering $\Delta_i$, must contain one of the following subgraphs
    \begin{figure}[H]
    \centering
    \includegraphics[scale = 0.7]{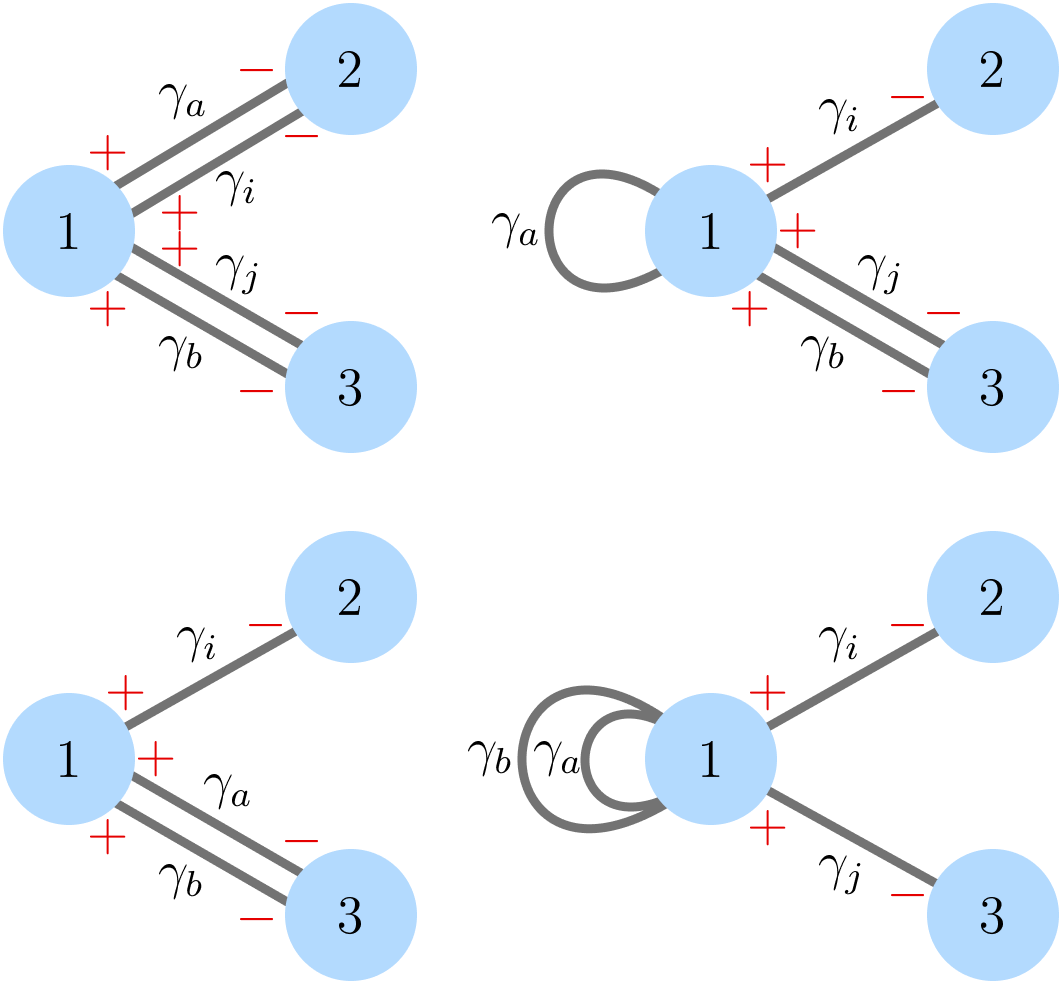}
    \end{figure}
    \noindent with $\{a,b\} = \{n,n-1 \}$.
    In the upper two cases we either have $$-\gamma_a -\gamma_b + \gamma_i + \gamma_j \in \Delta_2 \sqcup \Delta_3 \ \text{ or } \ 
    -\gamma_a -\gamma_b + \gamma_i + \gamma_j \in \Delta_1 \sqcup \Delta_3$$ respectively.
    These containments are incompatible with identities
    \begin{align*}
        (-\gamma_a -\gamma_b + \gamma_i + \gamma_j) + (-\gamma_i + \gamma_a) &= -\gamma_b + \gamma_j, \\
        (-\gamma_a -\gamma_b + \gamma_i + \gamma_j) + (-\gamma_j + \gamma_b) &= -\gamma_a + \gamma_i.
    \end{align*}
    For the bottom-left case $-\gamma_a -\gamma_b + \gamma_i + \gamma_j \in \Delta_3$ for all $1 \le j \le n-2$ different from $i$, while $\gamma_a + \gamma_b - \gamma_i \in \Delta_2$.
    This means that $\gamma_j \in \Delta_2 \sqcup \Delta_3$ which is impossible.
    Finally, in the last case $\gamma_a + \gamma_b - \gamma_i - \gamma_j \in \Delta_2 \sqcup \Delta_3$ is incompatible with
    \begin{align*}
        (\gamma_a + \gamma_b - \gamma_i - \gamma_j) + (-\gamma_a + \gamma_i) &= \gamma_b - \gamma_j, \\
        (\gamma_a + \gamma_b - \gamma_i - \gamma_j) + (-\gamma_a + \gamma_j) &= \gamma_b - \gamma_i.
    \end{align*}
\end{proof}

Now we are left with $5$ exceptional cases.
All of them can be reduced to either $B_n$ or $D_n$.

\subsubsection{Type $G_2$}
If $\alpha < \beta$ are two simple roots inside $G_2$, then the set of all roots looks as follows
$$
\Delta = \pm \{\alpha, \beta, \alpha + \beta, 2\alpha + \beta, 3\alpha + \beta, 3\alpha + 2\beta \}.
$$
The subset $S = \pm \{ \alpha, \beta, \alpha + \beta, 2\alpha + \beta \} \subseteq \Delta$
is not a subsystem, but it behaves in a way similar to $B_2$. 
More precisely, if we denote two simple roots of $B_2$ using the same letters $\alpha < \beta$, then $\lambda_1 \alpha + \lambda_2 \beta \in B_2$ for some $\lambda_i \in \mathbb{Z}$ if and only if $\lambda_1 \alpha + \lambda_2 \beta \in S$.
Since the roots of $B_2$ cannot be divided into more than two parts, the same is true for the roots in subset $S \subseteq \Delta$.
At the same time, the roots inside $S$ generate the whole system $G_2$.
Consequently, all the roots of $G_2$ can be partitioned in at most two parts.

\subsubsection{Type $F_4$}
Order simple roots as in \cref{fig:F4_string}.
We again consider the basis $\{ \beta_i \coloneqq \alpha_1 + \dots + \alpha_i \mid 1 \le i \le 4\}$.
The roots $\beta_1, \beta_2$ and $\beta_3$ generate a subsystem of $\Delta$ of type $B_3$,
\begin{figure}[h]
    \centering
    \includegraphics[scale = 0.7]{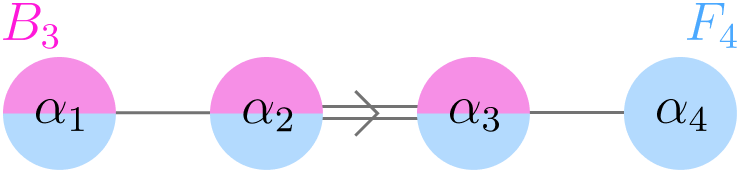}
    \caption{Subsystem $B_3$ inside $F_4$.}
    \label{fig:F4_string}
\end{figure}
which cannot be decomposed into more than $2$ parts.
Therefore, the graph of a $3$-regular partition of $\Delta$ with respect to the basis $\{\beta_i \}_1^4$ will necessarily have two vertices connected by the single edge $\beta_4$. Swapping the signs and re-numbering $\Delta_i$'s we may assume that $\beta_4 \in \Delta_1$ and $-\beta_4 \in \Delta_3$. 

If there is $1 \le i \le 2$ with $-\beta_i \in \Delta_2$, then we get a contradiction by writing
$$
\Delta_1 \ni \beta_4 = (2 \beta_4 - \beta_i) + (\beta_i - \beta_4) \in \Delta_2 \sqcup \Delta_3.
$$
Otherwise $-\beta_3 \in \Delta_2$ and the root $-\beta_4 - \beta_3 + \beta_i$, $1 \le i \le 2$, cannot be placed in $\Delta_2$ or $\Delta_3$.

\begin{corollary}
    Root system $F_4$ has no regular partitions into $m \ge 3$ parts.
\end{corollary}

\subsubsection{Type $E_n$}
Let us order the simple roots as it is shown in \cref{fig:En_string}.
\begin{figure}[h]
    \centering
    \includegraphics[scale = 0.7]{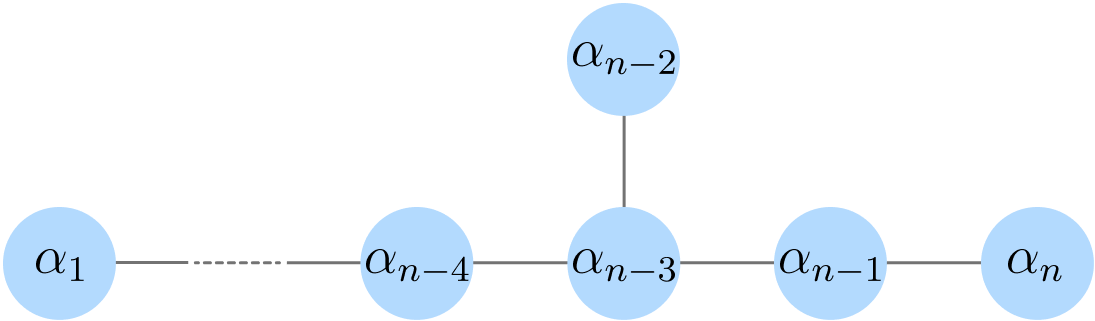}
    \caption{Ordering of simple roots in $E_n$ .}
    \label{fig:En_string}
\end{figure}
Consider the following basis for $\Delta$:
\begin{align*}
    \gamma_{i} &\coloneqq \alpha_1 + \dots + \alpha_i, \ 1 \le i \le n-2, \\
    \gamma_{n-1} &\coloneqq \alpha_1 + \dots + \alpha_{n-3} + \alpha_{n-1}, \\
    \gamma_{n} &\coloneqq \gamma_{n-1} + \alpha_n. 
\end{align*}
Subsystems of $E_n$ spanned by the sets $\{ \gamma_1, \dots, \gamma_{n-2}, \gamma_{n-1} \}$ and
$\{ \gamma_1, \dots, \gamma_{n-2}, \gamma_{n} \}$ are both equivalent to $D_{n-1}$ and, consequently, can be decomposed in at most two parts by \cref{prop:D_n_partition}.
As a result, up to changing the signs and re-numbering $\Delta_i$, the graph of a $3$-regular partition of $\Delta$ with respect to $\{ \gamma_i\}_1^n$ will be of the form
\begin{figure}[H]
    \centering
    \includegraphics[scale = 0.7]{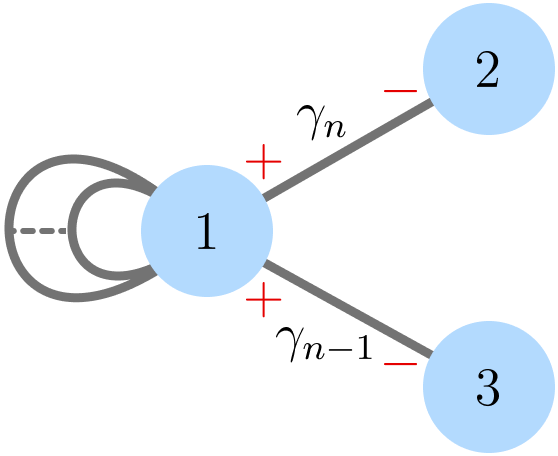}
\end{figure}
\noindent However, this is again impossible.
In all three cases, the root
\begin{align*}
    \gamma &\coloneqq -\gamma_n - \gamma_{n-1} - \gamma_{n-2} + \gamma_{n-3} + \gamma_{n-4} + \gamma_{n-5} \\ & \ = 
(-\gamma_n - \gamma_{n-2} + \gamma_{n-3} + \gamma_{n-4}) + (-\gamma_{n-1} + \gamma_{n-5})
\end{align*}
must lie inside $\Delta_2 \sqcup \Delta_3$.
This is impossible since none of these two contaiments is compatible with the equalities
\begin{align*}
    &\gamma + \underbrace{(\gamma_n + \gamma_{n-2} - \gamma_{n-3} - \gamma_{n-4})}_{\in \Delta_1} = \underbrace{-\gamma_{n-1} + \gamma_{n-5}}_{\in \Delta_2}, \\
    &\gamma + \underbrace{(\gamma_{n-1} - \gamma_{n-5})}_{\in \Delta_1} = \underbrace{-\gamma_n - \gamma_{n-2} + \gamma_{n-3} + \gamma_{n-4}}_{\in \Delta_3}.
\end{align*}
This observation results in the following statement.
\begin{corollary}
    Root system $E_n$ has no regular partitions into $m \ge 3$ parts.
\end{corollary}

\section{Regular decompositions of simple Lie algebras}
It was shown in \cref{prop:2_partititons}
that any $2$-regular partition of a finite-root system $\Delta$ can be extended to a $2$-regular decomposition of the corresponding Lie algebra $\fg(\Delta)$.
We now generalise this statement.

\begin{proposition}%
\label{prop:m_partititons}
    Given an $m$-regular partition $\Delta = S_1 \sqcup \dots \sqcup S_m$ with $m \ge 3$, we can find subspaces $\fs_1, \dots, \fs_m \subseteq \fh$ (not unique in general) such that
    $$
    \fg = \bigoplus_{i=1}^m \left( \fs_i \bigoplus_{\alpha \in S_i} \fg_\alpha \right)
    $$ 
    is a regular decomposition.
    Conversely, by forgetting the Cartan part of an $m$-regular decomposition of $\fg$ we obtain an $m$-regular partition of $\Delta$.
\end{proposition}
\begin{proof}
    To prove the statement it is enough to consider $\Delta$ of type $A_n$, $n \ge 2$, because regular partitions into $m \ge 3$ parts do not occur in other cases.
    By \cref{cor:any_decomposition_extends} any $m$-regular partition of $A_n$ can be obtained from one of the two finest ones
    $$\Delta_i = \pm \{-\beta_i + \beta_j \mid 0 \le i \neq j \le n \}, \ \ 0 \le i \le n$$
    by gluing together different elements $\Delta_i$.
    Each of these maximal partitions can be extended to an $(n+1)$-regular decomposition of $\fg$ as follows:
    \begin{align*}
        \fg_i &\coloneqq \textnormal{span}_F \{ E_{\alpha}, H_{\beta_i} \mid \alpha \in \Delta_i \}, \ \ 0 \le i \le n,  
    \end{align*}
where $H_{\sum c_i \alpha_i} \coloneqq \sum c_i H_{\alpha_i}$.
\end{proof}

It is clear that the distribution of Cartan elements in the proof above is not unique. 
For example, 
\begin{align*}
        \fg_i &\coloneqq \textnormal{span}_F \{ E_{\alpha}, H_{\beta_n - \beta_i} \mid \alpha \in \Delta_i \}, \ \ 0 \le i \le n
\end{align*}
is another possible extension.
Nevertheless, the Cartan part of $\fg$ must be distributed in a very restrictive manner. 
In particular, we can note that both extension above are the same up to the action of the Weyl group $W(A_n)$.

We now set $\fg = \mathfrak{sl}(n+1, F)$, $n \ge 2$ and show how $\mathfrak{h}$ can be placed inside an $m$-regular distribution.

\begin{lemma}%
\label{lem:one_Cartan_two_pieces}
Let $\fg = \oplus_{1}^m \fg_i$ be a regular decomposition of $\fg$.
If $\fh \subset \fg_j$ for some index $j$, then $m \le 2$.
\end{lemma}
\begin{proof}
For simplicity, assume $\fh \subset \fg_1$.
Take a root $\alpha$ and let $1 \le i, j \le m$ be two indices such that $$E_\alpha \in \fg_i \ \text{ and } \ E_{-\alpha} \in \fg_j.$$ 
On the one hand $[E_\alpha, E_{-\alpha}] \in \fh \subseteq \fg_1$, on the other hand $[E_\alpha, E_{-\alpha}] \in \fg_i \oplus \fg_j$.
Therefore, at least one of the indices $i$ or $j$ is equal to $1$. 
In other words, for each root $\alpha$ either $E_\alpha \in \fg_1$ or $E_{-\alpha} \in \fg_1$. 
By \cite[Proposition 20]{bourbaki_root_systems} this is equivalent to saying that $\fg_1$ is a parabolic subalgebra of $\fg$.
Moreover, acting on $\fg$ by its Weyl group, we can achieve that $$\mathfrak{n}_+ \coloneqq \bigoplus_{\alpha \in \Delta_+}\fg_\alpha 
\subset \fg_1.$$
\noindent
Let $\alpha_0$ be the maximal root of $\Delta$.
One of the following conditions must hold:
\begin{enumerate}
    \item $E_{-\alpha_0} \in \fg_1$, implying that $\fg = \fg_1$, or
    \item $E_{-\alpha_0} \in \fg_2$, producing the identity $\fg = \fg_1 \oplus \fg_2$.
\end{enumerate}
Here we have used the fact, that we can get any $E_{-\gamma} \in \mathfrak{n}_-$ by repeatedly commuting $E_{-\alpha_0}$ with $E_\alpha \in \fg_1$.
\end{proof}

\begin{corollary}%
\label{cor:nk=nn_nn+1}
For a regular decomposition of type $(m,k)$ we have $k \le \textnormal{rank}(\fg)$ and either 
$$
(m,k) = (k+1, k) \  \text{ or } \ (m,k) = (k,k).
$$
\end{corollary}
\begin{proof}
    Indeed, if $\fs_1, \dots, \fs_k \neq 0$, put $\fg_1' \coloneqq \fg_1 \oplus \dots \oplus \fg_k$ and apply \cref{lem:one_Cartan_two_pieces} to the regular splitting
    $$
    \fg = \fg_1' \bigoplus_{i = k+1}^m \fg_i .
    $$
\end{proof}

\begin{proposition}%
\label{prop:k_k_1_decomposition}
  Up to swapping positive and negative roots, re-numbering $\fg_i$ and the action of the Weyl group $W(A_n)$, any $(k+1, k)$-regular decomposition $\fg = \oplus_{1}^{k+1} \fg_i$ with $k \ge 2$ is of the form
  \begin{equation*}
        \begin{aligned}
            \fg_\ell &= \textnormal{span}_F \Bigg\{E_{-\beta_i + \beta_j}, H_{-\beta_i + \beta_n} \bigg| \sum_{t=1}^{\ell-1} \lambda_t \le i < \sum_{t=1}^{\ell} \lambda_t, \ 0 \le j \neq i \le n  \Bigg\}, \ 0 \le \ell \le k, \\
            \fg_{k+1} &= \textnormal{span}_F\Bigg\{E_{-\beta_n + \beta_j} \mid  0 \le j \neq i \le n  \Bigg\},
        \end{aligned}
    \end{equation*}
    where $\lambda = (\lambda_1, \dots, \lambda_{k}, 1)$ is an $(k+1)$-partition of $(n+1)$.
\end{proposition}
\begin{proof}
    Swapping positive and negative roots if necessary and applying \cref{cor:any_decomposition_extends}
    we see that there is a set partition $\{0, 1, \dots, n \} = S_1 \sqcup \dots \sqcup S_{k+1}$ such that
    $$
      \fg_\ell = \fs_\ell \bigoplus_{i \in S_\ell} \textnormal{span}_F\{E_{-\beta_i + \beta_j} \mid 0 \le j \neq i \le n\},
    $$
    for all $1 \le \ell \le k+1$.
    By re-numbering subalgebras $\fg_\ell$ we can assume $\fs_{k+1} = \{ 0 \}$.
    This, in particular, implies that $|S_{k+1}| = 1$ or, equivalentely,
    $$\fg_{k+1} = \textnormal{span}_F\{E_{-\beta_d + \beta_j} \mid 0 \le j \neq d \le n\},$$
    for some unique $0 \le d \le n$.
    Indeed, if $a,b \in S_{k+1}$ are different integers, then 
    $E_{\beta_a - \beta_b}, E_{\beta_b - \beta_a} \in \fg_{k+1}$ and hence $H_{\beta_a - \beta_b} \in \fg_{k+1}$ contradicting $\fs_{k+1} = 0$.

    For each $0 \le j \neq d \le n$, the Cartan element $H_{\beta_d - \beta_j}$ must be contained in the subalgebra $\fg_{\ell}$ that contains the root vector $E_{\beta_d - \beta_j}$.
    Since $\dim_F \fg_{k+1} = n$, this restriction distributes the whole Cartan in a unique way.
    Such a decomposition is well-defined, because for any $0 \le a \neq b \le n$
    we can write 
    $$
    H_{\beta_a - \beta_b} = H_{\beta_d - \beta_b} - H_{\beta_d - \beta_a}.
    $$
    
    By re-numbering the first $k$ subalgebras again we can achieve inequalities $|S_1| \ge |S_2| \ge \dots \ge |S_{k+1}|=1$.
    Acting on the decomposition with the Weyl group $W(A_n)$ we can order integers in $S_i$'s in the desired way:
    $$
    S_1 = \{0, 1, \dots, |S_1| - 1 \}, \ S_2 = \{ |S_1|, \dots, |S_1| + |S_2| - 1 \}, \dots, S_{k+1} = \{n\}.
    $$

\end{proof}

\begin{remark}
    The explicit expressions for $\fg_i$ in \cref{prop:k_k_1_decomposition} are different from the ones in \cref{mainthm:classification_algebra_decompostions}.
    However, they coincide if we just revert the order of simple roots; see \cref{rem:change_of_order}.
\end{remark}

Now let us look at a $(k,k)$-regular decomposition of $\fg$.
By the same argument as in the proof above
we can find a set partition $\{0, 1, \dots, n \} = S_1 \sqcup \dots \sqcup S_{k}$ such that
\begin{equation}%
\label{eq:explicit_form_g_i}
      \fg_\ell = \fs_\ell \bigoplus_{i \in S_\ell} \textnormal{span}_F\{E_{-\beta_i + \beta_j} \mid 0 \le j \neq i \le n\},
\end{equation}
for all $1 \le \ell \le k$.
By swapping positive and negative roots, re-numbering subalgebras $\fg_\ell$ and by acting with the Weyl group $W(A_n)$, we can assume that 
$$
    S_1 = \{0, \dots, \lambda_1 - 1 \}, \ S_2 = \{\lambda_1, \dots, \lambda_1 + \lambda_2 - 1 \}, \ \hdots \ , 
    S_k = \{\sum_1^{k-1} \lambda_i, \dots, n \},
$$
for some partition $\lambda = (\lambda_1, \dots, \lambda_k)$ of $n+1$.
Since $k < n+1$ we have $\lambda_1 > 1$.
Consequently, we have the inclusions
\begin{equation}%
\label{eq:Cartan_part_second_case}
\begin{aligned}
    FH_{\beta_1} \oplus \dots \oplus FH_{\beta_{\lambda_1 - 1}} &\subseteq \fs_1, \\
    FH_{\beta_{\lambda_1}} \oplus \dots \oplus FH_{\beta_{\lambda_1 + \lambda_2 - 1}} &\subseteq \fs_1 \oplus \fs_2, \\
     &\vdots \\
    FH_{\beta_{\sum_{1}^{k-1} \lambda_i}} \oplus \dots \oplus FH_{\beta_{n}} &\subseteq \fs_1 \oplus \fs_k.
\end{aligned}
\end{equation}
If for some $2 \le i \le k$ we have
$$
\bigoplus_{\ell \in S_i} FH_{\beta_{\ell}}  \subseteq \fs_1 \subseteq \fs_1 \oplus \fs_{i},
$$
then $\fs_{i} = \{0\}$ contradicting the assumption that each $\fs_j$ has positive dimension. 
Consequently, for all $2 \le i \neq j \le k$ we can find some $\ell \in S_i$, $d \in S_j$, $x_1, x_2 \in \fs_1$, $y \in \fs_i$ and $z \in \fs_j$
such that 
$$
x_1 + y = H_{\beta_\ell}, \ x_2 + z = H_{\beta_d}.
$$
Then $$H_{\beta_\ell - \beta_d} = (x_1 - x_2) + y - z \in \fs_1 \oplus \fs_i \oplus \fs_j.$$
At the same time $H_{\beta_\ell - \beta_d} \in [FE_{-\beta_d + \beta_\ell}, FE_{-\beta_\ell + \beta_d}] \subseteq \fs_i \oplus \fs_j$.
Therefore, we have the equality $x_1 = x_2$.
This observation leads to the following statement.

\begin{proposition}%
\label{prop:k_k_decomposition}
  Up to swapping positive and negative roots, re-numbering $\fg_i$ and the action of the Weyl group $W(A_n)$, any $(k, k)$-regular decomposition $\fg = \oplus_{1}^{k} \fg_i$ with $3 \le k \le n$ is of the form
  \begin{equation*}
        \begin{aligned}
            \fg_1 &= \textnormal{span}_F \Bigg\{E_{-\beta_i + \beta_j}, H_{\beta_i}, X \bigg| 0 \le i < \lambda_1, \ 0 \le j \neq i \le n  \Bigg\}, \\
            \fg_\ell &= \textnormal{span}_F \Bigg\{E_{-\beta_i + \beta_j}, H_{\beta_i} - X \bigg| \sum_{t=1}^{\ell-1} \lambda_t \le i < \sum_{t=1}^{\ell} \lambda_t, \ 0 \le j \neq i \le n  \Bigg\}, \ 2 \le \ell \le k, \\
        \end{aligned}
    \end{equation*}
    where $\lambda = (\lambda_1, \dots, \lambda_{k})$ is a $k$-partition of $(n+1)$ with $\lambda_1 > 1$ and $X$ is an arbitrary element in the set 
    $$
    (FH_{\beta_1} \oplus \dots \oplus FH_{\beta_{\lambda_1 - 1}}) \cup \left\{ H_{\beta_p} \mid 2 \le m \le k, \ \lambda_m > 1, \ \sum_{t=1}^{m-1} \lambda_t \le p < \sum_{t=1}^{m} \lambda_t  \right\}.
    $$
\end{proposition}
\begin{proof}
    Assume first that there is an index $\ell \in S_i$, $2 \le i \le k$, such that $H_{\beta_\ell} \in \fs_i$
    From the explicit form \cref{eq:explicit_form_g_i} we know that $H_{\beta_\ell - \beta_m} \in \fs_i $ for all $m \in S_i$.
    Therefore, $H_{\beta_m} \in \fs_i$ for all $m \in S_i$.
    Furthermore, combining this argument with the observation just before the proposition with $x_1 = 0$ shows that
    $$
    \bigoplus_{m \in S_j} FH_{\beta_m} = \fs_j
    $$
    for all $2 \le j \le k$. This gives the statement of the proposition with $X = 0$.
    
    Now we suppose that $H_\ell \not \in \fs_i$ for all $\ell \in S_i$ and $2 \le i \le k$.
    There are two cases to consider:
    \begin{itemize}
        \item All $H_\ell \not \in \fs_1$. Then each of them can be represented in the form $H_\ell = x_\ell + y_\ell$ for some non-zero $x_\ell \in \fs_1$ and $y_\ell \in \fs_i$.
        By the observation preceding the proposition all $x_\ell$ must be equal to a unique $X \in \fs_1$ and hence $$\bigoplus_{m \in S_i} F(H_{\beta_m}-X) = \fs_i.$$
        By counting dimensions we get  $$
        \fs_1 = \textnormal{span}_F \{X, H_{\beta_1}, \dots, H_{\beta_{\lambda_1 - 1}} \} = \textnormal{span}_F \{H_{\beta_1}, \dots, H_{\beta_{\lambda_1 - 1}} \}.
        $$
        \item There is an $2 \le i \le k$ and $\ell \in S_i$ such that $|S_i| \ge 2$ and $H_{\beta_\ell} \in \fs_1$.
        Such an element $H_{\beta_\ell}$ must be unique. Indeed, if there is another $H_{\beta_m} \in \fs_1$ for some $m \in S_j$, $2 \le j \le k$, then we get a contradiction
        $$
        \fs_1 \ni H_{\beta_\ell - \beta_m} \in [FE_{\beta_\ell - \beta_m}, FE_{\beta_m - \beta_\ell}] \in \fs_i \oplus \fs_j.
        $$
        Therefore, all $H_{\beta_m}$ with $\ell \neq m \in S_j$, $2 \le j \le k$ can be uniquely represented as $H_{\beta_m} = X + y_m$ for $X \in \fs_1$ and $y_m \in \fs_j$.
        Since $$
        H_{\beta_\ell - \beta_m} = \underbrace{(H_{\beta_\ell} - X)}_{\in \fs_1} - y_m \in \fs_i + \fs_j
        $$
        we have $X = H_{\beta_\ell}$.
    \end{itemize}
\end{proof}


\newpage
\printbibliography

\end{document}